\documentclass{amsart}
\usepackage{amsmath}
\usepackage{amsthm}
\usepackage{amsfonts}
\usepackage{amssymb}
\usepackage{amscd}
\usepackage{stmaryrd}
\usepackage[hidelinks]{hyperref}
\usepackage{enumerate}

\numberwithin{equation}{section}

\theoremstyle{plain}
\newtheorem{thm}[equation]{Theorem}

\newtheorem{cor}[equation]{Corollary}
\newtheorem{prop}[equation]{Proposition}

\theoremstyle{definition}
\newtheorem{definition}[equation]{Definition}

\newtheorem{example}[equation]{Example}

\newcommand{\Q}{\ensuremath \mathbb{Q}}

\newcommand{\Z}{\ensuremath \mathbb{Z}}

\begin{document}

\title[Odd, spoof perfect factorizations]{Odd, spoof perfect factorizations}

\author[BYU-CNT-Group]{BYU Computational Number Theory Group}
\thanks{A full list of authors may be found at: \url{https://math.byu.edu/~pace/spoof_authors.pdf}}


\begin{abstract}
We investigate the integer solutions of Diophantine equations related to perfect numbers.  These solutions generalize the example, found by Descartes in 1638, of an odd, ``spoof'' perfect factorization $3^2\cdot 7^2\cdot 11^2\cdot 13^2\cdot 22021^1$.  More recently, Voight found the spoof perfect factorization $3^4\cdot 7^2\cdot 11^2\cdot 19^2\cdot(-127)^1$.  No other examples appear in the literature.  We compute all nontrivial, odd, primitive spoof perfect factorizations with fewer than seven bases---there are twenty-one in total.

We show that the structure of odd, spoof perfect factorizations is extremely rich, and there are multiple infinite families of them.  This implies that certain approaches to the odd perfect number problem that use only the multiplicative nature of the sum-of-divisors function are unworkable.  On the other hand, we prove that there are only finitely many nontrivial, odd, primitive spoof perfect factorizations with a fixed number of bases.
\end{abstract}

\maketitle

\section{Introduction}

Let $\sigma$ denote the sum-of-divisors function.  A positive integer $n$ is said to be \emph{perfect} if $\sigma(n)=2n$; in other words, the sum of the proper divisors of $n$ equals $n$.  At present there are fifty-one perfect numbers known, all of the form $2^{p-1}(2^p-1)$ where both $p$ and $2^{p}-1$ are prime.  The smallest example is when $p=2$, and the current largest example is when $p=82589933$.

Euclid, in his \emph{Elements}, defined the perfect numbers and proved that if $2^p-1$ is prime (which necessitates $p$ being prime), then $2^{p-1}(2^p-1)$ is perfect.  Euler proved, conversely, that every even perfect number is of this form.  Two of the oldest open problems in mathematics are whether there are infinitely many perfect numbers, and whether any of them are odd.  This paper focuses on the second problem.

Many of the methods employed in the study of odd perfect numbers (hereafter denoted OPNs) apply to a much broader class of structures.  To motivate this generalization, we consider the (nonprime) factorization $n=3^2\cdot 7^2\cdot 11^2\cdot 13^2\cdot 22021^1$ discovered by Descartes.  Recall that $\sigma$ is multiplicative, and for any prime $p$ and any positive integer $a$ we have $\sigma(p^a)=1+p+\cdots+p^a$.  Thus, we might (falsely) compute
\begin{eqnarray*}
\sigma(n) & = & \sigma(3^2\cdot 7^2\cdot 11^2\cdot 13^2\cdot 22021^1) \\
& = & (1+3+3^2)(1+7+7^2)(1+11+11^2)(1+13+13^2)(1+22021)=2n.
\end{eqnarray*}
The problem is that $22021=19^2\cdot 61$, so $22021$ is not prime and the second equality above is false.  However, the given factorization of $n$ satisfies the condition for an OPN if we \emph{pretend} that $22021$ is prime and apply the usual rules for $\sigma$.  This motivates the following definitions.

\begin{definition}
Let $n\in \Z$.  We call an expression of the form $n=\prod_{i=1}^{k}x_i^{a_i}$, where each $x_i\in \Z$ and each $a_i\in \Z_{\geq 1}$, a \emph{factorization} of $n$.  We call each $x_i$ a \emph{base} of the factorization, and each $a_i$ is the corresponding \emph{exponent} of the $i$th base.  If, moreover, each $x_i\in \Z_{\geq 1}$, then the factorization is \emph{positive}.  A factorization is \emph{odd} when $n$ is odd; otherwise it is \emph{even}.

We define a function $\tilde{\sigma}$ on the collection of ordered pairs describing such a factorization by the rule
\[
\tilde{\sigma}\Big(\{(x_i,a_i):1\leq i\leq k\}\Big):=\prod_{i=1}^{k}\left(\sum_{j=0}^{a_i}x_i^{j} \right).
\]
As an abuse of notation, we will write $\tilde{\sigma}(\prod_{i=1}^{k}x_i^{a_i})$ instead of $\tilde{\sigma}(\{(x_i,a_i):1\leq i\leq k\})$.  We see that $\tilde{\sigma}(\prod_{i=1}^{k}x_i^{a_i})$ agrees with $\sigma(n)$ when the bases are distinct, positive primes.

A factorization as above is \emph{spoof perfect} if $\tilde{\sigma}(\prod_{i=1}^{k}x_i^{a_i})=2\prod_{i=1}^{k}x_i^{a_i}$.  (Note that the prime factorizations of actual perfect numbers are being treated as spoof perfect factorizations.  We use the word ``spoof'' rather than something like ``generalized'' for historical reasons.)
\end{definition}

\begin{example}
Descartes's example is a positive, odd, spoof perfect factorization.  Banks, G\"ulo\u{g}lu, Nevans, and Saidak\ \cite{Banks} searched for other positive, odd, spoof perfect factorizations of a form similar to Descartes's example, and Dittmer \cite{Dittmer} searched for an even more general class of positive, odd, spoof perfect factorizations.  Neither was successful in finding any additional examples.
\end{example}

\begin{example}
Voight in \cite{Voight} found the odd, spoof perfect factorization $3^4\cdot 7^2\cdot 11^2\cdot 19^2\cdot (-127)^{1}$, which is not positive.  Previous to this paper, Voight's and Descartes's examples were the only odd, spoof perfect factorizations to have been published.
\end{example}

\begin{example}\label{SpoofEvenPerfectFirstForm}
For each positive integer $a\in \Z_{\geq 1}$, the factorization $2^{a-1}\cdot(2^a-1)^1$ is even and spoof perfect.  Dittmer found other infinite families of even, spoof perfect factorizations in \cite{Dittmer}.
\end{example}

These examples suggest a natural dichotomy between the spoof perfect factorizations that are odd and those that are even.  We are able to prove that this dichotomy holds quite generally.  Our main result appears as Theorem \ref{MainTheorem}, which says that after restricting ourselves to what we will call nontrivial, primitive spoof perfect factorizations, there are only finitely many odd such factorizations with a fixed number of bases.  We find all twenty-one of them that have less than seven bases. These examples not only complement the examples of Descartes and Voight, but they provide natural barriers to certain proof strategies in trying to show that OPNs cannot exist.

An outline of the paper follows.  In Section \ref{Section:Trivialities} we characterize the spoof perfect factorizations with a single base, and also those with $0$ or $-1$ as a base.  The latter factorizations are called trivial, and throughout the rest of the paper only nontrivial factorizations are considered.  In Section \ref{Section:AbundDef} we apply the classical notions of abundant and deficient numbers to spoof perfect factorizations.  Using these concepts, we characterize the spoof perfect factorizations with two bases.

Next, in Section \ref{Section:2adic} we describe two additional tools that have been used in the study of OPNs, and apply them to this broader context of spoof perfect factorizations.  With these tools in hand, we give a (possibly complete) list of spoof perfect factorizations with three bases in Section \ref{Section:3base}.  Finally, in Section \ref{Section:OddSpoofsLots}, we characterize the odd, primitive spoof perfect factorizations with six or fewer bases, as well as prove our main theorem.

Along the way we discover a rich structure inherent in the odd, spoof perfect factorizations, including the fact that there are multiple infinite (nontrivial) families of such factorizations.  These infinite families necessarily have an increasing number of bases.  We discuss one final class of examples in Section \ref{Section:EndingQuestions}, where we also pose some additional open problems.

\section{Preliminaries and triviality}\label{Section:Trivialities}

One of the primary goals of this paper is to characterize many of the spoof perfect factorizations, in the hope that this information will be useful in the study of actual perfect numbers.  The spoof perfect factorizations with one base are easy to characterize.

\begin{prop}
The only spoof perfect factorization with one base is $1^1$.
\end{prop}
\begin{proof}
By inspection, $\tilde{\sigma}(1^1)=2\cdot 1^1$.

Suppose we have a general solution $\tilde{\sigma}(x^a)=2x^a$.  We then have
\[
1+x+x^2+\cdots + x^{a-1}=x^a.
\]
Clearly $x\neq 0$.  If $x\neq \pm 1$, then looking at this equation modulo $x$ we see there are no solutions.  When $x=-1$, if $a$ is even the equation becomes $0=1$, while if $a$ is odd the equation becomes $1=-1$.  Thus, the only possible solution is when $x=1$, which implies $a=1$.
\end{proof}

Before moving on to more bases, we first handle some trivialities.

\begin{prop}
Given a factorization $\prod_{i=1}^{k}x_i^{a_i}$, we have $\tilde{\sigma}(\prod_{i=1}^{k}x_i^{a_i})=0$ if and only if at least one of the bases is $-1$ and the corresponding exponent is odd.
\end{prop}
\begin{proof}
If $\tilde{\sigma}(x^a)=0$, then $x\neq 1$ and the geometric sum formula yields
\[
0=\sum_{j=0}^{a}x^j=\frac{x^{a+1}-1}{x-1}.
\]
Hence $x^{a+1}=1$, and thus $x=-1$ and $a$ is odd.

The converse is easy to verify.
\end{proof}

\begin{cor}\label{Cor:Trivial01}
A factorization where one of the bases equals $0$ is spoof perfect if and only if at least one other factor has base $-1$ with an odd exponent.
\end{cor}

We note that if $a$ is even, then $\tilde{\sigma}((-1)^a)=1$.  Thus, if we have a spoof perfect factorization, we can obtain a new spoof perfect factorization by adjoining (or, when possible, removing) a factor whose base is $-1$ and whose exponent is even.  In light of Corollary \ref{Cor:Trivial01} and this fact, we say that a factorization is \emph{trivial} if some base is $0$ or $-1$, otherwise it is \emph{nontrivial}.  For the remainder of the paper we will implicity assume all factorizations are nontrivial.

\section{Abundant and deficient factorizations}\label{Section:AbundDef}

Recall that the function $\sigma_{-1}(n)=\sum_{d|n}\frac{1}{d}$ is, just like $\sigma$, multiplicative.  In analogy to $\tilde{\sigma}$, we define (for nontrivial factorizations)
\[
\tilde{\sigma}_{-1}\Big(\{(x_i,a_i):1\leq i\leq k\}\Big)=\tilde{\sigma}_{-1}\left(\prod_{i=1}^{k}x_i^{a_i}\right) :=\prod_{i=1}^{k}\left(\sum_{j=0}^{a_i}\frac{1}{x_i^{j}} \right).
\]
Note that a factorization $\prod_{i=1}^{k}x_i^{a_i}$ is spoof perfect if and only if $\tilde{\sigma}_{-1}(\prod_{i=1}^{k}x_i^{a_i})=2$.

There are a few benefits to working with $\tilde{\sigma}_{-1}$ instead of $\tilde{\sigma}$.  First, when $x\in \Z\setminus \{-1,0,1\}$ we may set $\tilde{\sigma}_{-1}(x^{\infty}):=\frac{x}{x-1}$ because
\[
\lim_{a\to\infty}\tilde{\sigma}_{-1}(x^a)=\sum_{j=0}^{\infty}\frac{1}{x^j}=\frac{x}{x-1},
\]
thus allowing us to consider what happens to exponents ``at infinity.''  When $x=1$ we put $\tilde{\sigma}_{-1}(1^{\infty})=\infty$, and we continue to use the formula $x/(x-1)$ by treating $1/0$ as $\infty$.  Note that if we were to allow infinite exponents, there would be another spoof perfect factorization with one base, namely $2^{\infty}$.

More generally, note that since $\tilde{\sigma}_{-1}(x^{\infty})=\tilde{\sigma}_{-1}((x-1)^1)$ when $x\in\Z\setminus\{-1,0,1\}$, we can always replace infinite exponents with finite exponents in factorizations.  Thus, in our results characterizing spoof perfect factorizations with a given number of bases we will only list those factorizations involving finite exponents, leaving the reader to figure out the possible factorizations with infinite exponents, if they so desire.

Generalizing the literature, let us say that a factorization is \emph{deficient} if $\tilde{\sigma}_{-1}$ applied to the factorization yields an output that is less than $2$, and \emph{abundant} when the output is greater than $2$.  A quick calculation shows that an abundant factorization remains abundant if we adjoin an additional factor with positive base.  More generally, we have the following growth conditions.

\begin{prop}\label{Prop:AbundDefic}
{\bf (1)} If $x\in \Z_{\geq 1}$ and $a,b\in \Z_{\geq 1}\cup \{\infty\}$ with $a<b$, then
\[
1<\frac{x+1}{x}\leq \tilde{\sigma}_{-1}(x^a)<\tilde{\sigma}_{-1}(x^b)\leq \frac{x}{x-1}.
\]
Thus, for a fixed positive base, $\tilde{\sigma}_{-1}$ is strictly increasing as the exponent increases.  Moreover, the base determines the interval $\left[\frac{x+1}{x}, \frac{x}{x-1}\right]$ in which $\tilde{\sigma}_{-1}$ takes values.

{\bf (2)} If $1\leq x<y$ are integers, and $a,b\in \Z_{\geq 1}\cup\{\infty\}$, then
\[
\tilde{\sigma}_{-1}(x^a)\geq \tilde{\sigma}_{-1}(y^b)
\]
with equality only if $y=x+1$, $a=1$, and $b=\infty$.  Thus, on positive bases the function $\tilde{\sigma}_{-1}$ decreases as bases increase.

{\bf (3)} If $x\in \Z_{<-1}$, then
\[
\frac{1}{2}\leq \frac{x+1}{x}=\tilde{\sigma}_{-1}(x^1) < \tilde{\sigma}_{-1}(x^3)<\ldots <\frac{x}{x-1} <\ldots <\tilde{\sigma}_{-1}(x^4) < \tilde{\sigma}_{-1}(x^2)<1.
\]
Thus, the values of $\tilde{\sigma}_{-1}(x^a)$ oscillate around the limiting value $\tilde{\sigma}_{-1}(x^{\infty})=\frac{x}{x-1}$ when the base is negative.

{\bf (4)} If $y<x<-1$ are integers, and $a,b\in \Z_{\geq 1}\cup \{\infty\}$, then
\[
\tilde{\sigma}_{-1}(x^a)< \tilde{\sigma}_{-1}(y^b)
\]
apart from four cases.  For $x\in \Z_{<-1}$ and $n\in \Z_{\geq 1}$, those four cases are given by
\[
\begin{array}{rclrcl}
\tilde{\sigma}_{-1}(x^{2n}) & > & \tilde{\sigma}_{-1}((x-1)^1), \qquad & \tilde{\sigma}_{-1}(x^{\infty}) & = & \tilde{\sigma}_{-1}((x-1)^1),\\
\tilde{\sigma}_{-1}((-2)^{2}) & > & \tilde{\sigma}_{-1}((-3)^{2n-1}), \qquad & \tilde{\sigma}_{-1}((-2)^2) & = & \tilde{\sigma}_{-1}((-4)^1).
\end{array}
\]
\end{prop}
\begin{proof}
All of these are straightforward computations, left to the reader.
\end{proof}

We will now demonstrate the usefulness of these notions by characterizing the perfect factorizations with two bases.

\begin{prop}\label{Prop:TwoBase}
If $x^a\cdot  y^b$ is a nontrivial, spoof perfect factorization with $x\leq y$, then it is of the form $2^a\cdot (2^{a+1}-1)^{1}$ or it is one of the two sporadic solutions $(-2)^{1}\cdot 1^3$ and $(-3)^{1}\cdot 1^2$.
\end{prop}
\begin{proof}
If $x$ and $y$ are both negative then by part (3) of Proposition \ref{Prop:AbundDefic}, $\tilde{\sigma}_{-1}(x^a\cdot y^b)<1$, so the factorization isn't spoof perfect.

Next consider the case when $x$ and $y$ are both positive.  If $x=1$, then part (1) of Proposition \ref{Prop:AbundDefic} yields
\[
\tilde{\sigma}_{-1}(x^a\cdot y^b)>\tilde{\sigma}_{-1}(1^1)=2
\]
so the factorization is abundant.  If $x=2$, then
\[
\tilde{\sigma}_{-1}(y^b)=\frac{2}{\tilde{\sigma}_{-1}(x^a)}=\tilde{\sigma}_{-1}((2^{a+1}-1)^1).
\]
By part (2) of Proposition \ref{Prop:AbundDefic} we must have $y=2^{a+1}-1$, and then by part (1) we get $b=1$.  This gives us the infinite family of solutions given in the statement of the proposition.

Next, if $x\geq 3$ and $y\geq 4$, then we compute
\[
\tilde{\sigma}_{-1}(x^a\cdot y^b)<\tilde{\sigma}_{-1}(3^\infty\cdot 4^{\infty})=\frac{3}{2}\cdot\frac{4}{3}=2,
\]
so the factorization is deficient.  Thus the only other possible case with both $x$ and $y$ positive is $x=y=3$.  If $a=1$ or $b=1$, then the factorization is not perfect due to $2$-adic considerations.  (We look at general $p$-adic restrictions in the next section.)  Thus $a,b\geq 2$.  We then compute
\[
\tilde{\sigma}_{-1}(x^a\cdot y^b)\geq \tilde{\sigma}_{-1}(3^2\cdot 3^2)=\frac{169}{81}>2,
\]
so the factorization is abundant.

Finally, consider the case when $x\leq -2$ and $y$ is positive.  By part (3) of Proposition \ref{Prop:AbundDefic}, $\tilde{\sigma}_{-1}(x^a)<1$.  If $y\geq 2$, then $\tilde{\sigma}_{-1}(y^b)\leq 2$ and so $\tilde{\sigma}_{-1}(x^a\cdot y^b)<2$, hence the factorization is deficient.  So we reduce to the case when $y=1$.  If $b=1$, the factorization is deficient.  If $b\geq 4$, then we find
\[
\tilde{\sigma}_{-1}(x^a\cdot y^b)\geq \frac{1}{2}\tilde{\sigma}_{-1}(1^4)=\frac{5}{2}>2,
\]
so the factorization is abundant.  It is straightforward to check, using part (4) of Proposition \ref{Prop:AbundDefic}, that the remaining two cases give the sporadic solutions.
\end{proof}

\section{\texorpdfstring{$2$}{2}-adic valuation considerations}\label{Section:2adic}

Given that we want to solve $\tilde{\sigma}_{-1}(\prod_{i=1}^{k}x_i^{a_i})=2$, it is natural to look at the $2$-adic conditions necessary for this equality to hold.  These conditions are especially nice in the case where all the bases are odd.

\begin{prop}\label{Proposition:Eulerian}
In any odd, spoof perfect factorization, exactly one of the exponents in the factorization is odd.  If $x$ is the unique base with an odd exponent $a$, in such a factorization, then
\[
x\equiv a\equiv 1\pmod{4}.
\]
\end{prop}
\begin{proof}
Let $v_2$ denote the $2$-adic valuation on $\Q^{\times}$. If $x$ is any odd base, then $v_2(\tilde{\sigma}_{-1}(x^a))\geq 0$.  The valuation is strictly positive if and only if $a$ is odd.  As we want the $2$-adic valuation on the entire factorization to equal $1$, we see that exactly one base has an odd exponent. A direct computation through the cases mod $4$ shows that $v_2(\tilde{\sigma}_{-1}(x^a))=1$ if and only if $x\equiv a\equiv 1\pmod{4}$.
\end{proof}

This proposition restricts the class of potential spoof perfect factorizations, as it rules out many possible exponents on bases.  Euler was the first to prove and use this restriction, in the case of actual OPNs.  Thus, in any odd, spoof perfect factorization we will call the unique base with an odd exponent the \emph{Eulerian base}.

There is another result in the literature that applies to spoof perfect factorizations as well as actual perfect numbers; however, it applies only for positive, odd factorizations.

\begin{prop}\label{Prop:PaceBound}
If $n=\prod_{i=1}^{k}x_i^{a_i}$ is any positive, odd, spoof perfect factorization, then $n<2^{2^{2k}}$.
\end{prop}
\begin{proof}[Proof sketch.]
Follow the argument in \cite{NielsenBestBound}, but now for spoof perfect factorizations.
\end{proof}

\begin{cor}
There are finitely many positive, odd, spoof perfect factorizations with a fixed number of bases.
\end{cor}

While the previous proposition gives an effective bound on the size of any positive, odd, spoof perfect factorization in terms of the number of factors, it is often possible to further restrict cases by applying abundance and deficiency arguments.  Using such ideas, Dittmer \cite{Dittmer} showed that other than the spoof perfect factorization $1^1$ and Descartes's example, there are no other positive, odd, spoof perfect factorizations with fewer than seven bases.  In the remainder of this paper, we explore what happens if we weaken the positivity and parity conditions.  Surprisingly, many (but not all) of the finiteness conditions disappear, and a much richer structural pattern emerges.

\section{Three bases---more infinite families and complicated behavior}\label{Section:3base}

If we remove the restriction that bases are odd, there are many more spoof perfect factorizations.  We list below some spoof perfect factorizations having exactly three bases. (Exactly one of the given factorizations is odd.)  We found fifteen sporadic solutions.
\[
\begin{array}{lll}
\bullet\  (-10)^1\cdot (-3)^3\cdot 1^2\qquad & \bullet\  (-5)^1\cdot (-2)^3\cdot 1^3\qquad  & \bullet\  (-2)^2\cdot 1^1\cdot 3^1\\
\bullet\  (-9)^1\cdot (-4)^1\cdot 1^2 & \bullet\  (-4)^1\cdot (-3)^1\cdot 1^3 & \bullet\  3^1\cdot 4^1\cdot 5^1\\
\bullet\  (-7)^1\cdot (-3)^2\cdot 1^2 & \bullet\  (-3)^1\cdot (-2)^1\cdot 1^5 & \bullet\  3^1\cdot 4^2\cdot 7^1\\
\bullet\  (-6)^1\cdot (-5)^1\cdot 1^2 & \bullet\  (-3)^1\cdot (-2)^2\cdot 1^3 & \bullet\  3^3\cdot 4^2\cdot 35^1\\
\bullet\  (-5)^1\cdot (-2)^1\cdot 1^4 & \bullet\  (-2)^1\cdot (-2)^1\cdot 1^7 & \bullet\  3^3\cdot 5^1\cdot 8^1
\end{array}
\]
Letting $n\in \Z_{>0}$ serve as an index, there are also six infinite families.
\[
\begin{array}{ll}
\bullet\ (-2^{2n+2}+2^{n+2}-2)^1\cdot 2^n\cdot (2^{n+1}-1)^3 \qquad  &  \bullet\  (-2)^{2n-1}\cdot 1^2\cdot (2^{2n}-1)^1\\
\bullet\ (-2^{2n+2}+2^{n+1}-1)^1\cdot 2^n\cdot (2^{n+1}-1)^2  & \bullet\  (-n-1)^1\cdot 1^1\cdot n^1\\
\bullet\  (-2^{2n+1}-1)^1\cdot(-2)^{2n}\cdot 1^2  & \bullet\  3^1\cdot 3^n\cdot (3^{n+1}-1)^1
\end{array}
\]
Further, letting $m\in \Z_{>0}$ serve as a second index, there are three doubly-indexed infinite families.
\[
\begin{array}{l}
\bullet\ 2^n\cdot (2^{n+1})^{m}\cdot (2^{(n+1)(m+1)}-1)^1\\
\bullet\ \left(-(2^{n+1}-1)\left(\frac{2^{n+1}}{m}-1\right)\right)^1\cdot 2^n\cdot (2^{n+1}-1-m)^1, \text{ where }\, m|(2^{2n+2}-2^{n+1})\, \text{ and }\,m<2^{n+1}-2\\
\bullet\ 2^n\cdot(2^{n+1}+m)^1\cdot \left((2^{n+1}-1)\left(\frac{2^{n+1}}{m+1}+1\right)\right)^1,\text{ where }\, m|(2^{2n+2}-2^{n+1})\,\text{ and }\,m\leq 2^{n+1}-2
\end{array}
\]

These examples can be found employing abundance and deficiency computations and a case-by-case analysis.  Many cases can be eliminated by considering $p$-adic information, applying the following two useful results. In the following, $\Phi_n$ denotes the $n$th cyclotomic polynomial, $v_q(x)$ denotes the $q$-adic valuation of $x$ for an arbitrary prime $q$, and $o_q(x)$ denotes the order of $x$ modulo $q$.

\begin{prop}[{\cite[Theorems 94 and 95]{Nagell}}]
Let $q$ be a prime and $n\geq 1$ an integer.  Setting $k:=v_q(n)$, write $n=q^{k}m$.  The equation
\[
\Phi_n(x)\equiv 0\pmod q
\]
is solvable for some $x\in \Z$ if and only if $q\equiv 1 \pmod{m}$, and the solutions are exactly those integers $x$ with $o_q(x)=m$.  Moreover, if $x$ is such a solution, then
\[
v_q(\Phi_n(x))=\begin{cases}
v_q(x^n-1) & \text{ if }k=0,\\
1 & \text{ if } k\geq 2, \text{ or } k=1 \text{ and } n>2,\\
v_2(x+1) & \text{ if }k=1 \text{ and }n=2 \text{ \textup{(}so $q=2$\textup{)}.}
\end{cases}
\]
\end{prop}

\begin{cor}\label{Cor:CyclotomicAdic}
If $q\geq 3$ is a prime, $x\in \Z\setminus \{0\}$, and $a\in \Z_{>0}$, then
\[
v_q(\tilde{\sigma}_{-1}(x^a)) =
\begin{cases}
-a v_q(x) & \text{ if }q|x,\\
v_q(x^{o_q(x)}-1)+v_q(a+1) & \text{ if }o_q(x)|(a+1) \text{ and }o_q(x)\neq 1,\\
v_q(a+1) & \text{ if }o_q(x)=1,\\
0 & \text{ otherwise}.
\end{cases}
\]
A similar statement holds when $q=2$.
\end{cor}
\begin{proof}
Write
\[
\tilde{\sigma}_{-1}(x^a)=\frac{x^{a+1}-1}{x^a(x-1)}=\frac{\prod_{n|(a+1),\, n>1}\Phi_n(x)}{x^a}
\]
and use the previous proposition.
\end{proof}

Here is one example of how these results can eliminate candidate spoof perfect factorizations.  Consider a (supposed) spoof perfect factorization of the form $(-3)^a\cdot (-2)^b\cdot 1^3$ with $a\geq 2$ and $b\geq 4$.  Since it is spoof perfect we must have
\[
v_3(\tilde{\sigma}_{-1}((-2)^b))=-v_3(\tilde{\sigma}_{-1}((-3)^a\cdot 1^3))=a\geq 2.
\]
Corollary \ref{Cor:CyclotomicAdic} then implies that $3^2|(b+1)$.  Now, since $o_{19}(-2)=9$ and $9|b+1$, by the same corollary this forces $v_{19}(\tilde{\sigma}_{-1}((-2)^b))\geq 1$, so the $19$-adic valuation of the factorization is positive, yielding the needed contradiction.

We conjecture that our twenty-four bullet points above provide a complete list of the nontrivial, spoof perfect factorizations with three bases.  However, we are currently unable to rule out two additional cases.  First, writing the factorization as $x^a\cdot y^b\cdot z^c$ with $x\leq y\leq z$, we can't rule out the possibility that $x$ is negative, $y=2<z<2^{b+1}-1$, and one of $a$ or $c$ is not $1$.  The other case is similar, and occurs when $x=2$, $2^{a+1}<y\leq z$, and one of $b$ or $c$ is not $1$.

\section{Infinitely many nontrivial, odd, spoof perfect factorizations}\label{Section:OddSpoofsLots}

The three nontrivial, odd, spoof perfect factorizations with three or fewer bases are as follows.
\[
\begin{array}{l}
1^1\\
1^2\cdot (-3)^1\\
1^2\cdot (-3)^2\cdot (-7)^1.
\end{array}
\]
There is a pattern to these factorizations that continues indefinitely.  We can increase the exponent on the Eulerian base by one and then adjoin a new negative base, according to the formula
\begin{equation}\label{Eq:ImportantTrick}
\arraycolsep=1.4pt\def\arraystretch{2.8}
\begin{array}{rcl}
\tilde{\sigma}_{-1}(x^a) & = & \displaystyle{ \frac{x^a+x^{a-1}+\cdots+1}{x^a}= \frac{x^{a+1}+x^a+\cdots+1}{x^{a+1}}\cdot \frac{x^{a+1}+x^{a}+\cdots+x}{x^{a+1}+x^{a}+\cdots+1}}\\
& = & \displaystyle{\tilde{\sigma}_{-1}\left(x^{a+1}\cdot \left(-\sum_{j=0}^{a+1}x^j\right)^1\right)}.
\end{array}
\end{equation}
Notice that the new base $-\sum_{j=0}^{a+1}x^j$ is odd exactly when $a$ is odd, so this process gives a new odd, spoof perfect factorization that changes the Eulerian base.  Thus, the three factorizations above are part of an infinite family that continues
\[
\begin{array}{l}
1^2\cdot (-3)^2\cdot (-7)^2\cdot (-43)^1 \\
1^2\cdot (-3)^2\cdot (-7)^2\cdot (-43)^2\cdot (-1807)^1
\end{array}
\]
and so forth.

The two spoof perfect factorizations found by Descartes and Voight, respectively, do \emph{not} lie in the infinite family generated by applying this process to $1^1$, but each generates its own infinite family of odd, spoof perfect factorizations.

The ``base expansion trick'' encapsulated in (\ref{Eq:ImportantTrick}) can be applied to \emph{any} of the factors in an even spoof perfect factorization to give another such factorization.  Thus, in particular, each even infinite family with three bases gives rise to three new infinite families with four bases.

There are other (nontrivial) ways to get new spoof perfect factorizations from old ones.  Noting that $\tilde{\sigma}_{-1}((-3)^2\cdot 7^2\cdot 7^2\cdot (-19)^2)=1$, we can adjoin this product to any given spoof perfect factorization and obtain a new spoof perfect factorization with exactly four more bases.  To avoid this type of extension, we say that a spoof perfect factorization $\prod_{i=1}^{k}x_i^{a_i}$ is \emph{primitive} if for each proper subset $S\subsetneq \{1,2,\ldots, k\}$ the factorization $\prod_{i\in S}x_i^{a_i}$ is not spoof perfect.  Requiring primitivity, and bounding the number of bases, strongly limits the number of spoof perfect factorizations, at least among odd factorizations.

\begin{thm}\label{MainTheorem}
For each integer $k\geq 1$, there are finitely many nontrivial, odd, primitive spoof perfect factorizations with $k$ bases.
\end{thm}
\begin{proof}
We follow the ideas of \cite{Dittmer} and describe a finite process to construct all possible odd, primitive spoof perfect factorizations with at most $k$ bases.  Hereafter, $k$ is fixed and all spoof perfect factorizations we will consider have $\leq k$ bases and are odd.

First, define a \emph{partial factorization} to be an ordered set of triples of the form $(x_i,b_i,c_i)$ where:
\begin{itemize}
\item $x_i\neq -1$ is an odd integer, which we think of as one of the bases in the factorization,
\item $b_i$ is a positive integer, which we think of as a lower bound on the exponent of the base $x_i$, and
\item $c_i$ equals either $b_i$ or $\infty$; in either case we think of $c_i$ as an upper bound on the exponent.
\end{itemize}
Thus, for instance, the set $\{(3,2,\infty),(5,4,4),(-3,8,\infty)\}$ is a partial factorization which tells us that one of the bases is $3$ and the corresponding exponent is at least $2$, another base is $5$ and the corresponding exponent is exactly $4$, and a third base is $-3$ with corresponding exponent at least $8$.  Our ultimate goal is to recursively build up all \emph{complete} factorizations from \emph{partial} factorizations, for odd, primitive spoof perfect factorizations.

Next, we define a strict partial ordering on partial factorizations, which will give us a notion of \emph{improving} a partial factorization.  We put $\{(x_i,b_i,c_i)\}_{i=1}^{m}< \{(y_i,d_i,e_i)\}_{i=1}^{n}$ exactly when the following three conditions hold:
\begin{itemize}
\item[(1)] We have $m\leq n$.  This guarantees that the factorization cannot become shorter.

\item[(2)] For each $i\leq m$, we have $x_i=y_i$, $b_i\leq d_i$, and $c_i\geq e_i$.  Thus, the bases we have already chosen do not change, and the ranges on the previously chosen exponents can only tighten up.

\item[(3)] If $m=n$, then there exists some $i\leq m$ with $c_i=\infty$ and $d_i\neq \infty$.  This is to prevent the possibility of endlessly tightening the range of an exponent, such as
\[
\{(3,2,\infty)\}\to \{(3,4,\infty)\}\to\{3,6,\infty)\}\to \cdots.
\]
\end{itemize}

When improving a partial factorization we see that the new partial factorization always either contains an additional base, or the range on one of the exponents that was previously infinite will now be limited to a single positive integer.  Thus, any chain of improvements will have length at most $2k$, as we can only add at most $k$ bases, and can fix each exponent exactly once.  Hence, we just need a way to limit choices on exponents  and limit choices of bases to be adjoined.

Let $S=\{(x_i,b_i,c_i)\}_{i=1}^{m}$ be a partial factorization.  If $P=\prod_{i=1}^k x_i^{a_i}$ is any odd, primitive spoof perfect factorization that is compatible with $S$ (meaning $S\leq \{(x_i,a_i,a_i)\}_{i=1}^{k}$, and in particular $m\leq k$), then Proposition \ref{Prop:AbundDefic} gives upper and lower bounds on $\tilde{\sigma}_{-1}(\prod_{i=1}^{m}x_i^{a_i})$ using only the information in $S$.  Namely, the lower bound is
\[
L(S):=\tilde{\sigma}_{-1}\left(\prod_{i=1}^{m}x_i^{b_i'}\right)
\]
where
\[
b_i'=\begin{cases}
b_i & \text{if $x_i>0$, or $b_i$ is odd, or $b_i=c_i$},\\
b_{i}+1 & \text{otherwise}.
\end{cases}
\]
The upper bound $U(S)$ is defined similarly.  There are three cases we need to consider.

{\bf Case 1:} $L(S)>2$.  In this case we say that $S$ is abundant, since any factorization compatible with $S$ with the same number of factors is abundant.  To correct this defect, at least one of the remaining bases in $P$ must be negative.  Let $y$ be the largest (under the usual ordering on $\Z$) remaining negative base in $P$.  By Proposition \ref{Prop:AbundDefic}, parts (3) and (4), we have
\[
2=\tilde{\sigma}_{-1}(P)\geq L(S)\prod_{i=m+1}^{k}\tilde{\sigma}_{-1}(y^1)=L(S)\left(1+\frac{1}{y}\right)^{k-m}
\]
and solving the inequality for $y$ we reach
\begin{equation}\label{Eq:Limit1}
0>y\geq \frac{1}{\left(\frac{2}{L(S)}\right)^{1/(k-m)}-1}.
\end{equation}
In particular, we see that the next base in an extension of $S$ is forced to belong to a \emph{finite} interval.  Putting it another way, if we extend $S$ to a new partial factorization with exactly one new base, we may do so in only a finite number of ways, as limited by (\ref{Eq:Limit1}).

{\bf Case 2:} $U(S)<2$.  In this case we say that the partial factorization is deficient.  To correct this defect, at least one of the remaining bases in $P$ must be positive.  Let $y$ be the smallest remaining positive base in $P$. Applying Proposition \ref{Prop:AbundDefic}, parts (1) and (2), we have
\[
2=\tilde{\sigma}_{-1}(P)\leq U(S)\prod_{i=m+1}^k \tilde{\sigma}_{-1}(y^{\infty})=U(S)\left(\frac{y}{y-1}\right)^{k-m}
\]
and solving for $y$ we reach
\[
0<y\leq \frac{1}{1-\left(\frac{U(S)}{2}\right)^{1/(k-m)}}.
\]
Once again, this limits the next base to a \emph{finite} interval.

{\bf Case 3:} $L(S)\leq 2\leq U(S)$.  If $L(S)=U(S)=2$, we have a spoof perfect factorization, and any further extension will not be primitive.  Thus, we may reduce to considering the case when $L(S)<U(S)$, and in particular at least one of the elements of $S$ has $\infty$  as a third coordinate (else, from the definitions we gave for $L(S)$ and $U(S)$, their values would match).

Let $S_b$ be the partial factorization obtained from $S$ by replacing any triple $(x_i,b_i,c_i)$ where $c_i=\infty$ with the new triple $(x_i,b,b)$.  Note that $S<S_b$ as long as
\[
b\geq \max_{\{1\leq i\leq m\, :\, c_i=\infty\}}b_i.
\]
Further,
\[
\lim_{b\to \infty}L(S_b)=\lim_{b\to \infty}U(S_b);
\]
call this common value $n'$.  We now appeal to Proposition \ref{Proposition:Eulerian}, and note that $\tilde{\sigma}_{-1}(x^{\infty})=\frac{x}{x-1}$ has \emph{negative} $2$-adic valuation for any \emph{odd} $x$.  This implies that $n'\neq 2$.

Thus, there exists some (computable) positive integer $b$ such that $S<S_b$ and either $L(S_b)>2$ or $U(S_b)<2$.

Now, if $P$ is any odd, primitive spoof perfect factorization compatible with $S$, then one of two situations holds.  Either
\begin{itemize}
\item $P$ is also compatible with $S'$, where $S'$ is obtained from $S$ by replacing one of the triples $(x_i,b_i,\infty)$ with $(x_i,b_i',b_i')$, for one of the finitely many integers $b_i'\in [b_i,b]$, or
\item $P$ is compatible with $S_b$, so we may reduce to either Case 1 or Case 2 (after replacing $S$ with $S_b$) and adjoin one of finitely many new bases.
\end{itemize}

In every case, there are only finitely many improvements.  Further, as mentioned previously, the longest possible chain of improvements is $2k$, so there are only finitely many partial factorizations compatible with odd, spoof perfect factorizations.
\end{proof}

The process described in the proof of Theorem \ref{MainTheorem} can be turned into a program for finding all odd, spoof perfect factorizations with a given number of bases.  We implemented such a program in \textsc{Mathematica}, also including a subroutine incorporating the $2$-adic valuation conditions expressed in Proposition \ref{Proposition:Eulerian}.  The $k=5$ case terminated after about 30 minutes, yielding a complete list of the corresponding odd, primitive spoof perfect factorizations.  The $k=6$ case took considerably longer.  Over the space of two years we distributed the computation over multiple processes (ranging from between 20 to 80 processors at any given time), totaling over 30 processor years.

The corresponding computation for \emph{positive} factorizations, as done in \cite{Dittmer}, takes around 13.3 hours on a single processor.  A large majority of the extra time used in our computation was spent considering the cases where the first three bases are $3$, $5$, and $15$.  It may be possible to eliminate some of these cases without the need for a brute-force search.

We found a total of twenty-one nontrivial, odd, primitive spoof perfect factorizations with six or fewer bases.  Ten of these factorizations are listed below.  The other eleven can be constructed from those ten by (repeatedly) applying the formula (\ref{Eq:ImportantTrick}) to Eulerian factors, thus increasing the number of bases by one.
\begin{itemize}
\item[(1)] $1^1$
\item[(2)] $1^2\cdot (-3)^2\cdot (-5)^2\cdot 49^1$
\item[(3)] $1^2\cdot (-3)^2\cdot (-3)^2\cdot 7^2\cdot (-19)^1$
\item[(4)] $3^2\cdot 7^2\cdot 7^2\cdot 13^1\cdot (-19)^2$
\item[(5)] $3^2\cdot 7^2\cdot 11^2\cdot 13^2\cdot 22021^1$
\item[(6)] $3^4\cdot 7^2\cdot 11^2\cdot 19^2\cdot (-127)^1$
\item[(7)] $1^2\cdot (-3)^2\cdot (-3)^2\cdot 7^4\cdot (-17)^2\cdot 36413^1$
\item[(8)] $1^2\cdot (-3)^2\cdot (-5)^2\cdot 7^2\cdot (-7)^2\cdot (-2451)^1$
\item[(9)] $3^4\cdot 7^2\cdot 7^2\cdot (-19)^1\cdot 11^2\cdot  (-19)^2$
\item[(10)] $3^4\cdot 7^2\cdot 7^2\cdot (-19)^2\cdot 25^2\cdot (-3751)^1$
\end{itemize}

While each of these factorizations is primitive, and none arises from any other by repeated use of the base expansion trick in (\ref{Eq:ImportantTrick}), it is still sometimes possible to generate one from another.  To see this, first apply the base expansion trick to the Eulerian factor of $1^2\cdot (-3)^2\cdot (-5)^2\cdot 49^1$ to get the new ``derived'' spoof perfect factorization $1^2\cdot (-3)^2\cdot (-5)^2\cdot 49^2\cdot (-2451)^1$.  For any base $x\neq 0$, it happens that
\begin{equation}\label{Eq:SquareSquare}
\tilde{\sigma}_{-1}((x^2)^2)=\tilde{\sigma}_{-1}(x^2\cdot (-x)^2),
\end{equation}
so we can replace $49^2$ with $7^2\cdot (-7)^2$, and thus the eighth listed spoof perfect factorization arises in a natural way from the second.

We can us the same squared-square trick given in (\ref{Eq:SquareSquare}) to get another primitive spoof perfect factorization as follows.  Starting with $1^1$ as a seed, applying (\ref{Eq:ImportantTrick}) to the Eulerian factor three times in succession we have the spoof perfect factorization $1^2\cdot (-3)^2\cdot (-7)^2\cdot (-43)^1$.  Multiplying by $(-3)^2\cdot 7^2\cdot 7^2\cdot (-19)^2$, and then replacing $(-7)^2\cdot 7^2$ with $49^2$, we have the seven-base, primitive spoof perfect factorization
\[
1^2\cdot (-3)^2\cdot (-3)^2\cdot 7^2\cdot (-19)^2\cdot (-43)^1\cdot 49^2.
\]
Doing a partial search with $k=7$, our program found an additional seven-base, odd, spoof perfect factorization:
\[
1^2\cdot (-5)^2\cdot (-5)^2\cdot (-9)^2\cdot 7^2\cdot (-9)^2\cdot (-101251)^1.
\]

\section{More infinite families of odd, spoof perfect factorizations}\label{Section:EndingQuestions}

There are three main properties that can prevent an odd, spoof perfect factorization from corresponding to an actual OPN.  They are:
\begin{itemize}
\item The bases are allowed to be non-primes.
\item The bases are allowed to share prime factors.
\item The bases are allowed to be negative (and possibly zero if one allows trivial factorizations).
\end{itemize}
Thus, if any result on OPNs similarly does not use primality, relative primality, or positivity of the bases, then that result necessarily applies to the corresponding spoof perfect factorizations as well.

Consequently, Voight's spoof perfect factorization, which has prime bases that are pairwise relatively prime, shows us that any purported proof of the nonexistence of OPNs must \emph{necessarily} use the positivity of the bases in the factorization.  Similarly, Descartes's example shows that one must necessarily use the primality of the bases.  We were unable to find any example with only positive, prime bases (allowing for repetitions of bases).

This work raises the question of whether or not there is additional structure inherent in all spoof perfect factorizations, which in turn would limit the structure of OPNs.  For instance, for each of the spoof perfect factorizations listed above, the Eulerian exponent is always $1$.  Is this true in general?

It turns out that, no, the Eulerian exponent can be arbitrarily large for odd, spoof perfect factorizations.  For instance, consider the following 74-base factorization, where we use $[x^a]^b$ to mean that the factor $x^a$ is being repeated $b$ times:
\[
[(-619)^2]^4\cdot [(-31)^2]^7\cdot [(-19)^2]^2\cdot [(-11)^2]^6\cdot [(-7)^4]^{14}\cdot [7^2]^8\cdot 11^2\cdot [37^2]^6 \cdot [67^2]^5\cdot [163^2]^4 \cdot [191^2]^7\cdot [211^2]^2\cdot [2223^2]^8.
\]
If $P$ is this factorization, then $\tilde{\sigma}_{-1}(P)=1/3$.  Therefore, in the odd, spoof perfect factorization $1^5\cdot P$, the Eulerian exponent is $5$.  The odd, spoof perfect factorization $1^{17}\cdot[P]^2$, where we have repeated the factors in $P$ twice, has an even larger Eulerian exponent.  Repeating this process, we can make the Eulerian exponent arbitrarily large.

Other questions that we are currently unable to answer include:
\begin{itemize}
\item Are there only finitely many nontrivial, odd, (not necessarily primitive) spoof perfect factorizations with a given number of bases?  If so, can Proposition \ref{Prop:PaceBound} be modified to give an effective upper bound on the number of such factorizations?
\item Are there infinitely many nontrivial, odd, primitive spoof perfect factorizations that do not use $1$ as a base, and that do not arise from simpler factorizations using (\ref{Eq:ImportantTrick})?  If so, can we additionally guarantee that the bases are pairwise relatively prime?
\end{itemize}


\section*{Acknowledgements}

The project was sponsored by the National Security Agency under Grant Number H98230-16-1-0048.  This work was partially supported by a grant from the Simons Foundation (\#281876).

\providecommand{\bysame}{\leavevmode\hbox to3em{\hrulefill}\thinspace}
\providecommand{\MR}{\relax\ifhmode\unskip\space\fi MR }
\providecommand{\MRhref}[2]{%
  \href{http://www.ams.org/mathscinet-getitem?mr=#1}{#2}
}
\providecommand{\href}[2]{#2}

\newpage
\pagestyle{empty}

Here is a full list of the authors in the BYU Computational Number Theory Group for ``Odd, spoof perfect factorizations'':
\begin{itemize}
\item Nickolas Andersen
\item Spencer Durham
\item Michael J.\ Griffin
\item Jonathan Hales
\item Paul Jenkins
\item Ryan Keck
\item Hankun Ko
\item Grant Molnar
\item Kyle Niendorf
\item Eric Moss
\item Pace P.\ Nielsen
\item Vandy Tombs
\item Merrill Warnick
\item Dongsheng Wu
\end{itemize}
\noindent Correspondence should be directed to Pace P.\ Nielsen.
\\ Department of Mathematics, Brigham Young University, Provo, UT 84602, USA
\\ Email address: pace@math.byu.edu

\end{document}